\documentclass[12pt]{amsart}

\usepackage[english]{babel}
\usepackage[latin1]{inputenc}
\usepackage[T1]{fontenc}
\usepackage{indentfirst}
\usepackage{amsmath,amssymb,wasysym,latexsym,amsfonts,euscript,geometry}
\usepackage{amsthm}
\usepackage{mathpazo}
\usepackage[pdftex]{hyperref}
\usepackage{stmaryrd}
\newcommand{\C}{\mathbb C}
\newcommand{\R}{\mathbb R}
\newcommand{\Z}{\mathbb Z}

 \numberwithin{equation}{section}  \makeatletter\@addtoreset{equation}{section}
\newcommand{\sublaplatian}{{\mathcal L}_\omega}
 \newcommand{\e}{\operatorname{e}}
\newcommand{\Vol}{\mathbf{V}\!\operatorname{ol}}
\newcommand{\ext}{\operatorname{ext}}

\newcommand{\ent}[1]{{\left[{#1}\right]}}
\newcommand{\parent}[1]{{\left({#1}\right)}}
\newcommand{\modul}[1]{\left\vert#1\right\vert}
\newcommand{\abs}[1]{\left\vert#1\right\vert}
\newcommand{\set}[1]{\left\{#1\right\}}
\newcommand{\scal}[1]{\left<#1\right>}

\makeatletter \@addtoreset{equation}{section}
 \makeatletter\@addtoreset{equation}{subsection}

\newcommand{\overbar}[1]{\mkern 1.5mu\overline{\mkern-1.5mu#1\mkern-1.5mu}\mkern 1.5mu}
\newcommand{\scalforms}[1]{\left<#1\right>_{\Omega^p}}
\newtheorem {theorem}{Theorem}[section]
\newtheorem {definition}[theorem]{Definition}

\newtheorem {proposition}[theorem]{Proposition}
\newtheorem {remark}[theorem]{Remark}

\begin{document}

\title[On concrete spectral properties of a twisted-Laplacian]{On concrete spectral properties of a twisted-Laplacian associated to a central extension of the real Heisenberg group}
\thanks{This research work was partially supported by a grant from the Simons Foundation and by the Hassan II Academy of Sciences and Technology.
}
\author{Aymane EL Fardi}    \email{a.elfardi@gmail.com }
\author{Allal Ghanmi}       \email{ag@fsr.ac.ma}
\author{Ahmed Intissar}     \email{itissar@fsr.ac.ma}
 \address{A.G.S. - L.A.M.A, Department of Mathematics, P.O. Box 1014, Faculty of Sciences     \\ Mohammed V University of Rabat, Morocco}
%\date{\today}

\maketitle

\begin{abstract}
We consider the magnetic Laplacian $\Delta_{\nu,\mu}$ on $\R^{2n}=\C^n$ given by
$$
\Delta_{\nu,\mu}= 4\sum\limits_{j=1}\limits^{n}\frac{\partial^2 }{\partial z_j \partial \bar z_j} +2i\nu (E+ \overline{E} +n)  +2\mu (E- \overline{E} )  -(\nu^2+\mu^2)|z|^2.
$$
We show that $\Delta_{\nu,\mu}$ is connected to the sub-Laplacian of a group of Heisenberg type given by $\C\times_\omega \C^n$ realized as a
 central extension of the real Heisenberg group $H_{2n+1}$.
We also discuss invariance properties of $\Delta_{\nu,\mu}$ and give some of their explicit spectral properties.
\end{abstract}

\section{Introduction}

In the present paper we study the spectral properties of the second order % elliptic and symmetric
 differential operator
\begin{equation*}
 \Delta_{\nu,\mu} = { 4\sum\limits_{j=1}\limits^{n}\frac{\partial^2 }{\partial z_j\partial \bar
 z_j} +2(\mu+i\nu)E - 2(\mu-i\nu)\overbar E   -  ({\nu^2+\mu^2})|z|^2 + 2i\nu n},
\end{equation*}
acting on the free Hilbert space $ L^2(\C^n,dm) $, where $E$ is the Euler operator and $\overline{E}$ is its complex conjugate.
 The parameters $ \nu$ and $\mu $ are assumed to be real and $\mu>0$. The particular case of $\nu = 0$ and $\mu=2b$ with $ n=1 $  leads to minus four times the special Hermite operator (\cite{thangavelu-ectures,wong-weyl})
\[-4\mathbb{L}_b = 4 \left\lbrace \dfrac{\partial^2 }{\partial z\partial \bar z}+ b \left(z \dfrac{\partial}{\partial z }- \overbar{z}\dfrac{\partial}{\partial \overbar{z}}\right) - b^2 |z|^2\right\rbrace   .\]
Such operator is the Hamiltonian describing the quantum behavior of a charged particle on the configuration space $ \C^n$
 under the influence of a constant magnetic field \cite{asch-magnetic}. Geometrically, $\mathbb{L}_b$ represents a Bochner Laplacians $ \nabla^*\nabla $ on
 the smooth sections of a Hermitian line bundle with connection $ \nabla$ over the manifold $M=\C^n$ \cite{asch-magnetic,kuwabara-spectra}.

The main results to which is aimed this paper concern the realisation of $ \Delta_{\nu,\mu} $ as a magnetic Schr\"odinger operator associated
to a specific potential vector (Section 4). The connection to the sub-Laplacian of a group of Heisenberg type given by $\C\times_\omega \C^n$ is
also established (see Section 3). The group $\C\times_\omega \C^n$ is realized as a central extension $N_\omega= (\C\times \C^n,\cdot_{\omega}) $
of the standard Heisenberg group $H_{2n+1}=(\R\times \C^n,\cdot_{\Im m\omega}) $. In this new group, the symplectic form is extended and replaced
by an Hermitian product (details in Section 2).
Invariance properties of $ \Delta_{\nu,\mu} $ are discussed in Section 3 and concrete description of its $L^2$-spectral analysis
%, including .... ,
 is presented in Section 5.
In Section 6, we use the factorization method \cite{infeld-factorization,mielnik-factorization} to generate eigenfunctions of $ \Delta_{\nu,\mu} $ in
terms of multivariate version of complex Hermite polynomials. For the case of the twisted Laplacian of the standard Heisenberg group, one can refer to \cite{koch-spectral,tie-twisted}.

\section{The group $N_\omega =\C\times_\omega \C^n$ as a central extension of the Heisenberg group $H_{2n+1}=\R\times_{Im\omega}\C^n$}

We realize $N_\omega :=\C\times_\omega \C^n$ as a central extension of the Heisenberg group $H_{2n+1}:=\R\times_{Im\omega}\C^n$,
where $\omega(z,w)$ denotes the standard Hermitian form on $\C^n$. To this end, we follow the exposition given in \cite{manin-quantized}.
Being indeed, if $ (K,\bullet)$ and $ (G,\odot) $ are two abelian groups and $ \psi : K \times K \rightarrow G $ a given mapping.
 On $G\times K$ we define the $ \cdot_{\psi}$-law by
\[ (z_0;z)\cdot_{\psi} (w_0;w)=(z_0\odot w_0 \odot \psi(z,w); z \bullet w). \]
We say that $ G\times_{\psi} K $ is a central extension of $ (K,\bullet)$ by $ (G,\odot) $ associated to $ \psi $ if the short sequence
\[ 0 \rightarrow K \rightarrow G\times_{\psi} K \rightarrow G \rightarrow  0 \]
is exact, and such that $ K $ is in $ Z(G) $, the center of the group E.  This holds if one of the following two equivalent assertions is satisfied, to wit
\begin{itemize}
\item[i)] $\psi$ preserves the neutral element $\psi(0_K,0_K)=0_G $ and verifies the cocycle relation
\[ \psi(x, y)\odot \psi(x \bullet y,z)= \psi(x, y \bullet z)\odot \psi(y,z) \]
for every $ x,y,z \in K $.
\item[ii)] $G\times_{\psi} K :=(G\times K, \cdot_{\psi}) $ is a group.
\end{itemize}
Now, let $\R^2=\R_s\times \R_t$ be the real $(s,t)$-plane identified with the complex plane $\C=\set{z_0=s+it; ~ s,t \in \R}$ and $\C^n$ denotes the complex $n$-space endowed with its standard Hermitian form
$$\omega(z,w):=\scal{z,w}=\sum\limits_{j=1}\limits^nz_j\bar w_j$$
for $z=(z_1,z_2, \cdots, z_n)$ and $w=(w_1,w_2, \cdots, w_n)$ in $\C^n$.
Thus, we define $N_\omega=\C\times_\omega \C^n $ to be the set $\C\times \C^n$ endowed with the $\omega$-law given by
\begin{eqnarray} \label{omegalaw1}
(z_0;z)\cdot_\omega(w_0;w)=(z_0+w_0+ \scal{z,w}; z+w).
 \end{eqnarray}
Under \eqref{omegalaw1}, $N_\omega:= \C\times_\omega \C^n$ is a non-commutative nilpotent group of step two with center
$Z(N_\omega)=\C\times_\omega \{0\} = (\R\times\R)\times_\omega \{0\}$. The identity element is $ (0;0) $ and the symmetric
of an element $(z_0;z)$ is  $ (-z_0 - \scal{z,z}; - z) $.
Notice for instance that the $\omega$-law given by (\ref{omegalaw1}) can be rewritten in the coordinates $z_0=(s,t)$, $w_0=(s^{'},t^{'})$ and $z\in \C^n$
and $z,w\in \C^n$ as follows
\begin{eqnarray} \label{omegalaw2} %\label{eq:RealLow}
((s,t);z)\cdot_\omega((s^{'},t^{'});w)=((s+s^{'}+Re\scal{z,w}, t+t^{'}+Im\scal{z,w}); z+w).
\end{eqnarray}
Hence, endowing the set $\R_t\times\C^n$ with the $\cdot_{Im\omega}$-law given by
\begin{eqnarray} \label{omegalaw3}
(t;z)\cdot_{Im\omega}(t^{'};w)=(t+t^{'}+Im\scal{z,w}; z+w),
\end{eqnarray}
 makes $\R\times_{Im\omega}\C^n$ a group, which is nothing else than the classical real Heisenberg group of dimension $2n+1$.
 One can notice easily that $ (\C \times \C^n , \cdot_{\omega}) $, in addition of being the central extension of $ \C^n $ by  $ \C $ associated to the map $\psi = \omega $, can also be viewed, due to  \eqref{omegalaw2}, as the central extension of $ (\R_t \times \C^n , \cdot_{\Im m\omega}) $ by $ \R_s $ associated to $ \psi = \Re\omega $.
 This can be stated otherwise using directly the definition; if we denote by $q$ the projection mapping from $(\R_s\times\R_t)\times_{\omega}\C^n$ onto $\R_t\times_{Im\omega}\C^n$ given
 by $q(s,t;z)= (t;z)$, one check that the mapping $q$ is a homomorphism from the group $N_\omega=\C\times_\omega \C^n$ onto the Heisenberg group
$H_{2n+1}=\R\times_{Im\omega}\C^n$ and that the kernel of $q$ is given by
$$\ker q =\set{ (s,0;0), ~ s\in \R }= (\R_s\times \{0\})\times_{\omega}\{0\}.$$ Since $\ker q$ is contained in the center $Z(N_\omega)$
 of $N_\omega$, we may say that the group $N_\omega=\C\times_\omega \C^n$ is a central extension of the Heisenberg group
 $H_{2n+1}=\R\times_{Im\omega}\C^n$ by $(\R_s, +)$; i.e we have $\C\times_\omega \C^n/\ker q =\C\times_\omega \C^n/\R_s=H_{2n+1}$.
Accordingly, harmonic analysis on our group $N_\omega=\C\times_\omega \C^n$ will have many links to that on the classical Heisenberg group.

\section{Explicit formula for the sub-Laplacian ${\mathcal L}_\omega$ on $N_\omega=\C\times_\omega \C^n$}

 The group $N_\omega=\C\times_\omega \C^n$ with the $\omega$-law given in (\ref{omegalaw1}) is
 %clearly
 a real Lie group of
dimension $2n+2$, and its tangent space at its neutral element $e=(0;0)\in \C\times\C^n$ is given by $ T_{(0;0)}N_\omega=(\C,+)\times(\C^n,+)$ as a real vector space of dimension $2n+2$.
In fact, $ N_\omega $ is naturally equipped with the standard differentiable structure on euclidean spaces generated by the
 coordinates system $ \set{(\C\times\C^n,x)} $, where $ x $ is the coordinates map
 \[ %\begin{array}{rrl}
 x : \C\times\C^n\longrightarrow  \R^{2n+2} ;\quad (z_0,z) \longmapsto  (s,t,x_1,y_1,x_2,y_2,\cdots,x_n,y_n).
 %\end{array}
  \]
The group action and the group symmetric maps are smooth under this differentiable structure. Let denote by $ \mathfrak{n}_\omega $ its associated Lie algebra composed of all left invariant vector fields on $ N_\omega $ and endowed with the standard bracket on vector fields. It is a well known fact that $ \mathfrak{n}_\omega \cong T_{(0;0)}N_\omega $.
For the sack of giving the explicit formula for the sub-Laplacian ${\mathcal L}_\omega$ on $N_\omega=\C\times_\omega \C^n$,
% we need to determine the associated Lie algebra structure of $n_\omega$. To this end,
we need to build a basis of $ \mathfrak{n}_\omega $ which will be constructed as first order differential operators on functions of $N_\omega$.
Define the left action by a fixed element $ (z_0;z)\in N_\omega $ by
\[ %\begin{array}{rrl}
\ell_{(z_0;z)} : N_\omega\longrightarrow  N_\omega ;\quad (w_0;w) \longmapsto  (z_0;z) \cdot_\omega (w_0;w).
%\end{array}
\]
This map is a diffeomorphism with respect to the Lie group structure. Hence, it is possible to extend its push-forward to act on vector fields.
Furthermore, its action on a vector field $ X $ is given explicitly by
$${\ell_g}_* X_p f = X_p (f\circ\ell_g),$$
for test data $ p $ and $ f $ such that $ p \in N_\omega $ and $ f $ is a smooth function  of $ N_\omega $.
By definition, a vector field $ X $ is said to be left invariant if the equality  $ {\ell_g}_* X = X $ holds.

In order to construct a left invariant vector field basis, we take a basis of the tangent vectors at the identity
and generate from each vector of the tangent basis, a left invariant vector field by pushing it forward using $ \ell_{(z_0,z)} $.
Recall that a basis of the tangent vector space $ T_{(0;0)}N_\omega$ acting on smooth functions $f$ is given
by  \begin{equation}\label{eq:tangentbasis}
\parent{ \frac{\partial}{\partial x^i} }_{(0;0)} f := \partial_i(f\circ x^{-1})\bigg|_{x(0,0)}; \quad i=1,2,\cdots,2n+2,
\end{equation}
where $ \partial_i $ is the ordinary partial derivative with respect to the $i$-th variable.
We can now carry out the following computation in order to find generators for $ \mathfrak{n}_\omega $ :
\begin{align*}
{\ell_{(z_0;z)}}_* \parent{ \frac{\partial}{\partial x^i} }_{(0;0)} f = \parent{ \frac{\partial}{\partial x^i} }_{(0;0) } (f \circ \ell_{(z_0;z)} )
= \partial_i((f \circ \ell_{(z_0,z)})\circ x^{-1})\bigg|_{x(0,0)}
\end{align*}
We plug in $ x^{-1} \circ x $ in the middle of the last equation and we use the multivariable chain rule to get
\begin{align*}
{\ell_{(z_0;z)}}_* \parent{ \frac{\partial}{\partial x^i} }_{(0;0)} f & = \partial_i((f \circ x^{-1})\circ(  x \circ  \ell_{(z_0;z)}\circ x^{-1}))\bigg|_{x(0,0)}\\
&=\sum_{m=1}^{2n+2} \partial_m (f \circ x^{-1})\bigg|_{ x \circ \ell_{(z_0;z)} \circ x^{-1}\circ x(0,0)} \times \partial_i (x^m \circ \ell_{(z_0;z)} \circ x^{-1} ) \bigg|_{x(0,0)}\\
&= \sum_{m=1}^{2n+2} \mathbf{J}_{m,i} \times \parent{ \frac{\partial}{\partial x^m} }_{(z_0;z)},
\end{align*}
where $ \mathbf{J}_{m,i} := \partial_i (x^m \circ \ell_{(z_0;z)} \circ x^{-1} ) \bigg|_{x(0,0)} $ %can be viewed as the components of the following Jacobian matrix
 and $ x^m \circ \ell_{(z_0;z)} \circ x^{-1} $ is the $ m $-th coordinate map of $ x \circ \ell_{(z_0;z)} \circ x^{-1} $.
Explicitly, we have
\begin{align*}\textstyle
 x \circ & \ell_{(z_0;z)} \circ x^{-1}(s',t',x_1',y_1',\cdots,x_n',y_n')= \\
 & \left(s+s'+\sum\limits_{j=1}^{n}(x_jx_j'+ y_jy_j'),t+t'+\sum_{j=1}^{n}(y_jx_j' - x_jy_j'), x_1+x_1',y_1+y_1',\cdots,y_n+y_n' \right)
\end{align*}
Therefore, it follows that $ \mathbf{J}_{m,i} $ can be viewed as the components of the following Jacobian matrix
\begin{align*}
\mathbf{J}:=\mathbf{J}|_{(0,\cdots,0)} &= {\begin{pmatrix}
1 & 0 & x_1 & y_1 & \cdots & x_n & y_n  \\
0 & 1 & y_1 & -x_1 & \cdots & y_n & -x_n \\
\vdots & 0 & \ddots & & & & 0 \\
\vdots & &  & \ddots& & &  \\
\vdots & &  & &\ddots & &  \\
 \vdots & &  & & &\ddots &  \\
0 & 0 & \cdots & & &\cdots & 1
\end{pmatrix}}.
\end{align*}
Reading vertically, column by column, we find the following basis
\begin{eqnarray}  \left\{ \begin{array}{ll}
S =\parent{\frac{\partial }{\partial s}}\\
T =\parent{\frac{\partial }{\partial t}}\\
X_j =  x_j\parent{\frac{\partial }{\partial s}}+  y_j \parent{\frac{\partial }{\partial  t}} + \parent{\frac{\partial }{\partial x_j}}\\
Y_j= y_j \parent{\frac{\partial }{\partial s}} - x_j \parent{\frac{\partial }{\partial t}} + \parent{\frac{\partial }{\partial y_j}} .
\end{array} \right. \label{vectorfields26}
\end{eqnarray}
Note that we are  using the coordinates $z_0=s+it$ and $z_j=x_j+iy_j$ with %, and we change the notations of the basis in \eqref{eq:tangentbasis} as follows
\begin{equation}\label{eq:convention}
\frac{\partial }{\partial x^1} = \frac{\partial }{\partial s}, \quad \quad \frac{\partial }{\partial x^2}
= \frac{\partial }{\partial t}, \quad \quad \frac{\partial }{\partial x^{2j+1}} = \frac{\partial }{\partial x_{j}},
\quad \quad \frac{\partial }{\partial x^{2j+2}} = \frac{\partial }{\partial y_{j}},
\end{equation}
for $j=1,\cdots, n$.
%We make the following statement where we summarize the above discussion with some additional remarks.
We summarize the above discussion on $N_\omega=\C\times_\omega \C^n$ %with the group law
   %associated to the $\omega$-law $$ (z_0;z)\cdot_\omega(w_0;w)=(z_0+w_0+ \omega(z,w); z+w); ~~ \omega(z,w)=\scal{z,w}$$
and its associated Lie algebra $ \mathfrak{n}_\omega $ with some additional remarks by making the following statement.

\begin{proposition} \label{Proposition21}
The real vector fields $S=\frac{\partial }{\partial s}$,   $T=\frac{\partial }{\partial t}$ together with
  $X_j$, $Y_j$; $ j=1,\cdots,n $ %following the convention in \eqref{eq:convention}
  given by
\begin{eqnarray*}
X_j =  x_j\parent{\frac{\partial }{\partial s}}+  y_j \parent{\frac{\partial }{\partial  t}} + \parent{\frac{\partial }{\partial x_j}} \quad \mbox{and} \quad
Y_j= y_j \parent{\frac{\partial }{\partial s}} - x_j \parent{\frac{\partial }{\partial t}} + \parent{\frac{\partial }{\partial y_j}}
\end{eqnarray*}
form a basis for $ \mathfrak{n}_\omega $. %$=T_{(0;0)}N_\omega$.
Moreover, they satisfy the following commutation relations of Heisenberg type
\begin{eqnarray*}
\left\{ \begin{array}{lllll} ~ [S,X_j] = [S, Y_j] = 0 \\    ~ [
T,X_j] = [T, Y_j] = 0 \\  ~ [S,T] = 0 \\ ~ [X_j,X_k] = [Y_j,Y_k] =
0
\\ ~ [X_j,Y_k] = -2\delta_{jk} T & \mbox{ for all} j,k= 1, \cdots ,
n.\end{array}\right.
\end{eqnarray*}
\end{proposition}

\begin{remark}\label{Remark21}
As expected we see, in view of the above proposition, that the Lie algebra $ \mathfrak{n}_\omega $ of $N_\omega=\C\times_\omega \C^n$ with $\omega(z,w)=\scal{z,w}$ is
also a central extension of the classical Heisenberg algebra $H_{2n+1}=\R\times_{Im\omega}\C^n$ generated by the vector fields
$$
\set{ T=\frac{\partial }{\partial  t}, \tilde{X}_j=-y_j\frac{\partial }{\partial  t} + \frac{\partial }{\partial x_j},
\tilde{Y}_j=x_j\frac{\partial }{\partial t}+ \frac{\partial }{\partial y_j}}; \quad j=1, \cdots, n,
$$
 with the nontrivial commutation relation $[\tilde{X}_j,\tilde{Y}_k] = -2 T $, where $(x_j,y_j)$; $j=1, \cdots, n,$, are coordinates of $\C^n$.
\end{remark}

 \begin{remark}
To build such left invariant vector fields, one can also look for a one parameter group of $N_\omega$, i.e.,
a group homomorphism $\gamma:(\R,+) \longrightarrow N_\omega$ (curves $\gamma(\varepsilon)\in N_\omega$; $\varepsilon\in \R$) satisfying
$$\dot\gamma(0)=\frac{d\gamma}{d\varepsilon}(\varepsilon)|_{\varepsilon=0} =(v_0;v)\in T_{(0;0)}N_\omega=\C\times\C^n .$$
  \end{remark}

Next, we define in below the sub-Laplacian by setting

\begin{definition}\label{Definition21}
Let $X_j,Y_j$; $j=1, \cdots, n$, be the vector fields given in Proposition \ref{Proposition21}. Then, the operator
 $${\mathcal L}_\omega = \sum\limits_{j=1}\limits^{n}X_j^2+Y_j^2$$
 is called here the sub-Laplacian of $N_\omega=\C\times_\omega \C^n$.
\end{definition}

 The following proposition gives the explicit differential expression of ${\mathcal L}_\omega$ in terms of the Laplace-Beltarmi $\Delta_{\R^{2n}}$
 of $\C^n=\R^{2n}$ and the first order differential operators $E_{x,y}$ and $F_{x,y}$ defined by
 $$ \Delta_{\R^{2n}} := \sum\limits_{j=1}\limits^{n}\frac{\partial^2 }{\partial x_j^2}+ \frac{\partial^2 }{\partial y_j^2}; \quad
    E_{x,y}:= \sum\limits_{j=1}\limits^{n} x_j \frac{\partial}{\partial x_j} + y_j \frac{\partial}{\partial y_j}; \quad
    F_{x,y}:= \sum\limits_{j=1}\limits^{n} x_j \frac{\partial}{\partial y_j} - y_j \frac{\partial}{\partial x_j}
 .$$
 Namely, we have

\begin{proposition}\label{Proposition22}
The sub-Laplacian ${\mathcal L}_\omega$  as defined in the above definition is given explicitly in the coordinates $t,s,x_j,y_j,$ $j=1, \cdots, n,$ of $N_\omega=\C\times_\omega\C^n$ as follows
\begin{eqnarray} \label{sublaplacien27}
{\mathcal L}_\omega = \Delta_{\R^{2n}} + 2(E_{x,y} + n) \frac{\partial}{\partial s} - 2F_{x,y} \frac{\partial}{\partial t}
+ (|x|^2+|y|^2)(\frac{\partial^2}{\partial s^2}+\frac{\partial^2 }{\partial t^2}),
\end{eqnarray}
where $|x|^2=\sum\limits_{j=1}\limits^{n} x_j^2$ and $|y|^2=\sum\limits_{j=1}\limits^{n} y_j^2$.
\end{proposition}

The explicit expression of $\sublaplatian$ given in Proposition \ref{Proposition22} can be handled by straightforward computations.

\begin{remark}\label{Remank22} If we consider the coordinates $(s,t)\in \R^2=\C$ and $z=(z_1, \cdots ,z_n)\in \C^n$ with $z_j=x_j+iy_j$, then the
sub-Laplacian ${\mathcal L}_\omega$   in \eqref{sublaplacien27} can be rewritten as
\begin{eqnarray}
{\mathcal L}_\omega = 4\sum\limits_{j=1}\limits^{n}\frac{\partial^2 }{\partial z_j\partial\bar z_j} +2 (E+ \overline{E} +n)\frac{\partial }{\partial s} - 2 i (E- \overline{E} )
\frac{\partial }{\partial t} + |z|^2(\frac{\partial^2 }{\partial s^2}+\frac{\partial^2 }{\partial t^2}), \label{sublaplacien28}
\end{eqnarray}
where $E=\sum\limits_{j=1}\limits^{n}z_j \frac{\partial }{\partial z_j}$ is the complex Euler operator and $ \overline{E} =\sum\limits_{j=1}\limits^{n}\bar z_j \frac{\partial }{\partial \bar z_j}$  is its complex conjugate.
 \end{remark}

\begin{remark}\label{Remark23} The action of ${\mathcal L}_\omega$ on functions $F(t;z)$ on $N_\omega=\C\times_\omega \C^n$
 that are independent of the argument $s$, reduces to that of the sub-Laplacian
\begin{eqnarray}
\widetilde{{\mathcal L}}_{Im\omega} = 4\sum\limits_{j=1}\limits^{n}\frac{\partial^2 }{\partial z_j\partial \bar z_j}  - 2 i (E- \overline{E} ) \frac{\partial }{\partial t} + |z|^2 \frac{\partial^2 }{\partial t^2} \label{sublaplacien29}
\end{eqnarray}
of the classical Heisenberg group $\R\times_{Im\omega}\C^n=H_{2n+1}$.
\end{remark}

We conclude this section by mentioning that both operators ${\mathcal L}_\omega$ and $\widetilde{{\mathcal L}}_{Im\omega}$
are  not elliptic. But they are not far from being such in many aspects of their spectral theory. We will make this precise by discussing in a concrete manner the spectral eigenfunction problem on $\C^n$ of the
associated elliptic differential operator %$\Delta_{\nu,\mu}$ associated to ${\mathcal L}_\omega$ and that are given by
\begin{align}
 \Delta_{\nu,\mu}
 & = 4\sum\limits_{j=1}\limits^{n}\frac{\partial^2 }{\partial z_j\partial \bar z_j}
 + 2i\nu (E+ \overline{E} +n)  +2\mu (E- \overline{E} )  -(\nu^2+\mu^2)|z|^2 \nonumber\\
 &= 4\sum\limits_{j=1}\limits^{n}\frac{\partial^2 }{\partial z_j\partial \bar z_j} +2(\mu+i\nu)E -2(\mu-i\nu) \overline{E}   -
\left(\nu^2+\mu^2 \right)|z|^2 +2i\nu n . \label{laplacian31}
\end{align}
Formally, $ \Delta_{\nu,\mu}$ is related to ${\mathcal L}_\omega$  using partial Fourier transform in $(s,t)$ with $(i\nu,i\mu)$ as dual arguments.

In the next section, we see that the operator $\Delta_{\nu,\mu}$ can also be regarded as Schr\"odinger operator on $\C^n=\R^{2n}$
in the presence of a uniform magnetic field $\overrightarrow{B}_\mu=  i d\theta_{\nu,\mu}$ on $\C^n=\R^{2n}$
associated to a specific differential $1$-form  $\theta_{\nu,\mu}$.

\section{Realization of $ \Delta_{\nu,\mu}$ as a magnetic Schr\"odinger operator and invariance property}

Magnetic Schr\"odinger operator on a complete oriented Riemannian manifold $(M,g)$ %on scalar functions
is defined to be
 \begin{eqnarray}\label{eq:shrodinger1}
 H_{\theta} =(d+ \operatorname{ext} \ \theta)^{*}(d+ \operatorname{ext} \ \theta) ,
 \end{eqnarray}
 where $ \theta $ is a given $ \mathcal{C}^{1} $ real differential $ 1 $-form on $ M $ (potential vector).
  Here $d$ stands for the usual exterior derivative acting on the space of differential $p$-forms $ \Omega^p (M)$, $\operatorname{ext}\theta$ is the
 operator of exterior left multiplication by $\theta$, i.e., $(\operatorname{ext}\theta) \omega = \theta \wedge \omega$
 % for all  differential $p$-form $\omega$ of $M$
 and $(d+ \mbox{ext}\theta)^{*}$ is the formal adjoint of  $d+ \mbox{ext}\theta$ with respect to the Hermitian product
 \[ \scalforms{\alpha , \beta} = \int_{M} \alpha \wedge \star \beta \]
  induced by the metric  $ g $ on $ \Omega^p = \Omega^p (M)$, where $ \star $ denotes the Hodge star operator associated to the volume form.
 From general theory of Schr\"odinger operators on non-compact manifold $M$  (see for example \cite{shubin-essential}), it is known that the operator $ H_{\theta} $, viewed as an unbounded operator  in $ L^2(M,dm) $,
 is essentially self-adjoint for any smooth measure $ dm $.

 In our framework $ M$ is the complex  $ n $-space $ \C^n $ equipped with its K\"ahler metric
 $$
  g= ds^2=- \frac{i}{2} \sum\limits_{j=1}\limits^{n}dz_j\otimes d \overbar z_j= \sum\limits_{j=1}\limits^{n}dx_j\otimes dy_j
  $$
and the corresponding volume form is $\Vol=dx_1 dy_1 \cdots dx_n dy_n$.
Associated to the parameters $\nu$ and $\mu$, we consider the potential vector
 \begin{align}
 \theta_{\nu,\mu}(z)
 &:=-\frac{\mu-i\nu}{2}\sum\limits_{j=1}\limits^{n}\bar z_jd z_j +\frac{\mu+i\nu}{2} \sum\limits_{j=1}\limits^{n} z_jd\bar z_j.
 \label{theta32}
\end{align}

 Thus, we prove the following result concerning the twisted Laplacian defined by \eqref{laplacian31}.

 \begin{proposition}\label{prop:schrodinger}
 For every complex-valued $\mathcal{C}^\infty$ function $f$ on $\C^n$, we have
\begin{align}\label{laplacian33}
\Delta_{\nu,\mu} f= - H_{\theta_{\nu,\mu}}f = - (d+ \operatorname{ext}\theta_{\nu,\mu})^{*}(d+ \operatorname{ext}\theta_{\nu,\mu}) f.
\end{align}
 \end{proposition}

 \begin{proof}[Sketched proof]
 We start by writing $H_{\theta_{\nu,\mu}} :=(d+ \operatorname{ext}\theta_{\nu,\mu})^{*}(d+ \operatorname{ext}\theta_{\nu,\mu})$ as
\begin{align*}
 H_{\theta_{\nu,\mu}} =d^*d f+d^*\ext{\theta_{\nu,\mu}}f+(\ext{\theta_{\nu,\mu}})^*df+(\ext{\theta_{\nu,\mu}})^*\ext{\theta_{\nu,\mu}}.
\end{align*}
Next, using the well-known facts $ d^*~=~-\star d\star $ and $ (\operatorname{ext}\theta)^*=\star\operatorname{ext}\theta\star $, we establish the following
\begin{align*}
& d^*d=-4\sum\limits_{j=1}^{n}\frac{\partial}{\partial{z_j}}\frac{\partial}{\partial\overbar{z_j}},\\
& d^*\ext{\theta_{\nu,\mu}}f+(\ext{\theta_{\nu,\mu}})^*d = \sum\limits_{j=1}^{n} \parent{-2(\mu+i\nu)
z_j\frac{\partial}{\partial z_j}+2(\mu-i\nu)\overbar{z_j}\frac{\partial}{\partial\overbar{z_j}} - 2 i\nu} ,\\
& (\ext{\theta_{\nu,\mu}})^*\ext{\theta_{\nu,\mu}} =(\mu^2+\nu^2)\modul{z}^2 .
\end{align*}
\end{proof}

One of the advantages of the formula for $\Delta_{\nu,\mu}$ as given by \eqref{laplacian33}
with the differential $1$-form $\theta_{\nu,\mu}$ in \eqref{theta32} is that we can derive easily some invariance properties of the Laplacian $\Delta_{\nu,\mu}$  with respect
to the group of rigid motions of the complex Hermitian  space $(\C^n, ds^2)$; $ds^2= \sum\limits_{j=1}\limits^{n}dz_j\otimes d \bar z_j$.
Let $G$ denote the group of biholomorphic mapping of $\C^n$ that preserve the Hermitian metric
$ds^2$. Then, $G=U(n) \ltimes  \C^n$ is the group of semi-direct product of the unitary group $U(n)$ of $\C^n$ with the additive group $(\C^n, +)$.
It can be represented as
\begin{equation}
G:= U(n) \ltimes \C^n = \set{g=\begin{pmatrix}
A & b \\
0 & 1
\end{pmatrix} =:[A,b]; \ \ A\in U(n), b\in \C }
\end{equation}
and acts transitively on $\C^n$ via the mappings
$% \begin{equation}
 g.z = Az + b .%; \qquad g=(A,b)\in U(n) \rtimes  \C^n, z\in \C^n.
 % \label{action} \end{equation}
 $
The pull-back $g^{*}\theta_{\nu,\mu}$ of the differential $1$-form $\theta_{\nu,\mu}$ by the above mapping $z \longmapsto g.z$ is related to $\theta_{\nu,\mu}$
by the following identity for every $g\in G= U(n) \ltimes  \C^n$.

\begin{proposition}\label{Claim2}
Let $\theta_{\nu,\mu}$ be as in \eqref{theta32} and $g\in G= U(n) \rtimes  \C^n$. Then, for every $g\in G=U(n) \rtimes  \C^n$  we have
\begin{equation}\label{eq:invariantp}
 g^* \theta_{\nu,\mu}  =\theta_{\nu,\mu} + \frac{d j^{\nu,\mu}(g,z)}{j^{\nu,\mu}(g,z)},
\end{equation}
where
\begin{equation}\label{eq:autFact}
j^{\nu,\mu}(g,z)=\exp(i \phi_{\nu,\mu}(g,z)).
\end{equation}
The phase function $\phi_{\nu,\mu}(g,z)$ is given by
\begin{eqnarray}
\phi_{\nu,\mu}(g,z)= -\nu Re \left(\scal{z,g^{-1}.0}\right)+ \mu Im\left(\scal{z,g^{-1}.0}\right). \label{phasefunction35}
\end{eqnarray}
\end{proposition}

\begin{proof}
The identity \eqref{eq:invariantp} holds by component-wise straightforward computations. Indeed, direct computation yields
\begin{align*}
g^{*}\theta_{\nu,\mu}(z)
 & = \theta_{\nu,\mu}(z)
    - \frac{i\nu}{2} d\left[\scal{z,g^{-1}.0}+ \overbar{\scal{z,g^{-1}.0}} \right]
    + \frac{\mu}{2} d\left[\scal{z,g^{-1}.0}-\overline{\scal{z,g^{-1}.0}}\right]\\
   & =  \theta_{\nu,\mu}(z)  + id(\phi_{\nu,\mu}(g,z)),
\end{align*}
where $g^{-1}$ is the inverse mapping of $z \longmapsto g.z$ and $g^{-1}.0= -A^{-1}b=-A^{*}b$ for $g=[A,b]\in U(n) \ltimes  \C^n$. We conclude since
$ d j^{\nu,\mu}(g,z) = id(\phi_{\nu,\mu}(g,z)) j^{\nu,\mu}(g,z)$.
\end{proof}

 Notice that the relation \eqref{eq:invariantp} reads also as $g^* \theta_{\nu,\mu}  =\theta_{\nu,\mu} + d \log(j^{\nu,\mu}(\gamma,z))$ and
  shows that the differential $1$-form $\theta_{\nu,\mu}$ is not $G$-invariant. But $g^{*}\theta_{\nu,\mu}$ and
 $\theta_{\nu,\mu}$ are in the same class of the de Rham cohomology group. Also it gives insight how to make, in view of
the expression (\ref{laplacian33}), the Laplacian $\Delta_{\nu,\mu}$ invariant with respect to a $G$-action on
functions built with the help of the following automorphic factor $j^{\nu,\mu}(g,z)$ defined through \eqref{eq:autFact} and
satisfying the chain rule
\begin{equation}\label{chainrule}
j^{\nu,\mu}(gg',z) = j^{\nu,\mu}(g,g'z)j^{\nu,\mu}(g',z)
\end{equation}
for every $g\in G=U(n) \ltimes  \C^n$ and $z\in \C^n$.
Associated to $j^{\nu,\mu}$, we define $T^{\nu,\mu}_{g}$ to be the operator acting on differential $p$-forms $\omega$ of $\C^n$ through the formula
\begin{eqnarray*}
T^{\nu,\mu}_{g}\omega= j^{\nu,\mu}(g,z) g^{*} \omega . %\label{action36}
\end{eqnarray*}
On $\mathcal{C}^\infty$ complex-valued functions $f$ on $\C^n$, it reduces further to
\begin{eqnarray}
[T^{\nu,\mu}_{g}f](z)= j^{\nu,\mu}(g,z) g^{*} f (z) = j^{\nu,\mu}(g,z) f (g.z).
\label{action36}
\end{eqnarray}
Thus, the following invariance property for $\Delta_{\nu,\mu}$ holds.

\begin{proposition}\label{Claim3}
For every $g\in U(n) \ltimes  \C^n$, we have
\begin{eqnarray}
 \Delta_{\nu,\mu} T^{\nu,\mu}_{g}= T^{\nu,\mu}_{g} \Delta_{\nu,\mu}.
\label{invariance}
\end{eqnarray}
\end{proposition}

\begin{proof}
Using the well-known facts $g^{*}d=dg^{*}$ and $g^{*}(\alpha\wedge \beta)= g^{*}\alpha\wedge g^{*}\beta$, we get
\begin{align*}
T^{\nu,\mu}_g\big((d+ \operatorname{ext} \theta_{\nu,\mu}) f\big)
%= j^{\nu,\mu}(\gamma,z)  \Big[g^{*}\Big((d+ \operatorname{ext}\theta_{\nu,\mu})f\Big)\Big]
= j^{\nu,\mu}(\gamma,z)  \Big(d[g^{*}f]+ [g^{*}\theta_{\nu,\mu} ]\wedge [g^{*}f]\Big).
\end{align*}
Now, by means of the identity \eqref{eq:invariantp}, it follows
\begin{align*}
T^{\nu,\mu}_g\big((d+ \operatorname{ext} \theta_{\nu,\mu}) f\big)
 &=  j^{\nu,\mu}(\gamma,z) d [g^{*}f]+ j^{\nu,\mu}(\gamma,z)\theta_{\nu,\mu} [g^{*}f] +  d(j^{\nu,\mu}(\gamma,z)) [g^{*}f] \\
&= d(j^{\nu,\mu}(\gamma,z) [g^{*}f])  + \theta_{\nu,\mu} j^{\nu,\mu}(\gamma,z)[g^{*}f]\\
&= (d+ \operatorname{ext} \theta_{\nu,\mu}) \big(T^{\nu,\mu}_g f\big).
\end{align*}
Moreover, $T^{\nu,\mu}_g$ commutes also with $(d+ \operatorname{ext} \theta_{\nu,\mu})^{*}$ for $T^{\nu,\mu}_g$
being a unitary transformation. Therefore, by means of the expression of $ \Delta_{\nu,\mu} =-
 (d+ \operatorname{ext} \theta_{\nu,\mu})^{*}(d+ \operatorname{ext} \theta_{\nu,\mu})$ as a magnetic Schr\"odinger operator
 $ H_{\theta_{\nu,\mu}} $, we deduce easily that $\Delta_{\nu,\mu}$ and $T^{\nu,\mu}_g$ commute. This ends the proof.
\end{proof}

\begin{remark}\label{Remark}
For $g\in \{I\} \ltimes  \C^n=(\C^n,+)$, the unitary operators $T^{\nu,\mu}_{g}$ given in  (\ref{action36}) define
 projective representation of $ G $ on the space of $ \mathcal{C}^{\infty} $ functions on $ \C^n $.
In fact, they are the so-called magnetic translation operators that arise in the study of Schr\"odinger operators
in the presence of uniform magnetic field.
\end{remark}

\section{Spectral properties of $\Delta_{\nu,\mu}$ acting on $\mathcal{C}^\infty(\C^n)$ and on $\mathcal{H}=L^2(\C^n,dm)$}\label{section:spectral1}

We denote by $\C^n$ the Frechet space of complex-valued functions on $\mathcal{C}^\infty(\C^n)$ endowed with the
compact-open topology, while $L^2(\C^n,dm)$ denotes the usual Hilbert space of square integrable complex-valued functions $F(z)$ on $\C^n$
with respect to the usual Lebesgue measure $dm(z)$.
In the sequel, we will give a concrete description of the eigenspaces of $\Delta_{\nu,\mu}$ in both  $\mathcal{C}^\infty(\C^n)$ and $L^2(\C^n,dm)$.
 To this end, let $\lambda$ be any complex number in $\C$ and $E_\lambda(\Delta_{\nu,\mu})$ be the eigenspace of $\Delta_{\nu,\mu}$ corresponding to the eigenvalue
$-2\mu(2\lambda+n)$ in $\mathcal{C}^\infty(\C^n)$, i.e.,
\begin{eqnarray}
E_\lambda(\Delta_{\nu,\mu})=\set{F\in \mathcal{C}^\infty(\C^n); ~ \Delta_{\nu,\mu} F  = -2\mu(2\lambda+n) F}.\label{eigenspace37}
\end{eqnarray}
 Also, by $A^2_\lambda(\Delta_{\nu,\mu})$ we denote the subspace of $L^2(\C^n,dm)$ whose elements $F(z)$ satisfy $\Delta_{\nu,\mu} F = -2\mu(2\lambda+n) F$.
 Namely, by elliptic regularity of $\Delta_{\nu,\mu}$, we have
  \begin{eqnarray} \label{L2eigenspace38}
  A^2_\lambda(\Delta_{\nu,\mu}) %: = \set{F\in L^2(\C^n,dm); ~ \Delta_{\nu,\mu} F=-2\mu(2\lambda+n) F}
   = L^2(\C^n,dm) \cap E_\lambda(\Delta_{\nu,\mu}).
\end{eqnarray}

The first result related to $E_\lambda(\Delta_{\nu,\mu})$ and $A^2_\lambda(\Delta_{\nu,\mu})$ is the following.

\begin{proposition}\label{Claim4}
The eigenspaces $E_\lambda(\Delta_{\nu,\mu})$ and $A^2_\lambda(\Delta_{\nu,\mu})$ are invariants under the $T^{\nu,\mu}$-action given by
\eqref{action36}.
\end{proposition}

 \begin{proof}
This can be handled easily making use the invariance property \eqref{invariance} of $\Delta_{\nu,\mu}$ by the unitary transformations $T^{\nu,\mu}_g$.
 \end{proof}

\begin{proposition}\label{Claim5}
The set of spherical eigenfuctions of $\Delta_{\nu,\mu}$ with $-2\mu(2\lambda+n)$ as eigenvalue is a one dimensional
 vector subspace of $E_\lambda(\Delta_{\nu,\mu})$ generated by
\begin{eqnarray} \label{radialsolution}
\varphi^{\nu,\mu}_\lambda(z)=  e^{-\frac{\mu-i\nu}{2}|z|^2} {_1F_1}(-\lambda; n; \mu|z|^2)
\end{eqnarray}
where ${_1F_1}(a; c; x)$ is denoting here the usual confluent hypergeometric function,
$$ {_1F_1}(a; c; x)= 1+ \frac{a}{c}\frac{x}{1!} + \frac{a(a+1)}{c(c+1)}\frac{x^2}{2!}  +   \cdots ; \quad x\in \C.$$
  \end{proposition}

 \begin{remark}
 By a ``spherical" (or radial here) eigenfuction of $\Delta_{\nu,\mu}$, we mean $U(n)$-invariant function $f$ satisfying
 $f(h.z)=f(z)$ for all $h\in U(n)$ and $z\in \C^n$.
 \end{remark}

 \begin{proof}[Sketched proof]
 To prove the statement, we write $\Delta_{\nu,\mu}$ in polar coordinates $z=r\theta$ with $r\geq 0$ and $\theta\in S^{2n-1}$ as
$$
\Delta_{\nu,\mu} = \frac{\partial^2}{{\partial r}^2 } + \left(\frac{2n-1}{r} +2i\nu\right) \frac{\partial}{\partial r}
 -(\nu^2+\mu^2) r^2  + 2i\nu n  + L_{\nu,\mu}^\theta,
$$
where $ L_{\nu,\mu}^\theta$ stands for the tangential component of $\Delta_{\nu,\mu} $.
The eigenvalue problem $\Delta_{\nu,\mu} f = -2\mu(2\lambda+n) f $ for radial functions $f(z)=\psi(x)$, with $x=r^2$, reduces to
the differential equation
\[
\left\{x\frac{\partial^2}{{\partial x}^2 } + (n+i\nu x ) \frac{\partial}{\partial x}
-[\frac{\nu^2+\mu^2}{4} x + \frac{i\nu - \mu}{2} n  - \mu\lambda ] \right\} \psi = 0 .
\]
 Next, making use of the appropriate change of function $\psi(x)=e^{\frac{i\nu-\mu}{2}x} y(x) $, we see that the previous equation leads to
 the confluent hypergeometric differential equation \cite[page 193]{NO1978Fr}
\[
xy^{''}+(n-\mu x)y^{'}+\mu \lambda y = 0
\]
whose regular solution at $x=0$ is the confluent hypergeometric function  ${_1F_1}(-\lambda;n;\mu x)$.
\end{proof}

\begin{remark}\label{Remark}
According to the proof of the previous result, we make the following key observation that can deserve as outline of the proofs of Proposition~\ref{Claim5} and the assertions below. Indeed, the operators $\Delta_{\nu,\mu}$ and $\Delta_{0,\mu}$ are unitary equivalent in $L^2(\C^n,dm)$. More precisely, we have
 $$\Delta_{\nu,\mu}= e^{-\frac{i\nu}{2}|z|^2}\Delta_{0,\mu}e^{+\frac{i\nu}{2}|z|^2}.$$
\end{remark}

Accordingly, we claim the following

\begin{proposition}\label{Claim6}
Let $(\nu,\mu)\in \R^2$ with $\mu>0$ and $\lambda\in \C$. Then, the eigenspace $A^2_\lambda(\Delta_{\nu,\mu})$ as defined (\ref{L2eigenspace38})
is non-zero (Hilbert) space if and only if $\lambda=l$ with $l=0,1,2, \cdots $, is a positive integer number.
Moreover, the spaces $A^2_l(\Delta_{\nu,\mu})$, $l=0,1,2, \cdots $, are pairwise orthogonal in  $L^2(\C^n,dm)$ and
we have the following orthogonal decomposition in Hilbertian subspaces
$$L^2(\C^n,dm)= \bigoplus\limits_{l=0}\limits^{\infty} A^2_l(\Delta_{\nu,\mu}).$$
\end{proposition}

\begin{remark}\label{RemarkSpectrum}
The claim \ref{Claim6} asserts that the spectrum of $\Delta_{\nu,\mu}$ in $L^2(\C^n,dm)$  is purely discrete and each
of its eigenvalue  $-2\mu(2l+n)$, $l\in \Z^+$, is independent of $\nu$ and occurs with infinite degeneracy, i.e., the
eigengspace $A^2_l(\Delta_{\nu,\mu})$ in (\ref{L2eigenspace38}) is of infinite dimension.
\end{remark}

  \begin{proposition}\label{Claim7}
Let $(\nu,\mu)\in \R^2$ with $\mu>0$. For fixed $l=0,1,2, \cdots $, let $P_l$  be the orthogonal eigenprojector operator from $L^2(\C^n,dm)$ onto the eigenspace
$A^2_l(\Delta_{\nu,\mu})$ with $-2\mu(2l+n)$ as eigenvalue. Then the Schwartz kernel $P^{\nu,\mu}_l(z,w)$ of the operator $P_l$ is
given by the following explicit formula
\begin{eqnarray}
P^{\nu,\mu}_l(z,w)= %C_n(\mu;l)
(\frac{\mu}{\pi})^n \frac{(n-1+l)!}{(n-1)!l!} j^{\nu,\mu}(z,w)e^{-\frac{\mu}{2}|z-w|^2 } {_1F_1}(-l; n; \mu|z-w|^2)   ,
\label{Schwartzkernel39}
\end{eqnarray}
 where the factor $j^{\nu,\mu}(z,w)$; $z,w\in \C^n$ is given by
\begin{eqnarray}
j^{\nu,\mu}(z,w) =  e^{-\frac{i\nu}{2}(|z|^2-|w|^2)+\frac{\mu}{2}(\scal{z,w}-\overline{\scal{z,w}}) }
. \label{factor311}\end{eqnarray}
\end{proposition}

 \begin{proof}[Sketched proof]
 The proof for $\nu=0$ is contained in \cite{AvronHerbestSimon-Duke1978,AskourIntissarMouayn-JMP2000,GI-JMP2005}. For arbitrary $\nu$, the proof can be handled in a similar way or making use of the key observation that in $L^2(\C^n,dm)$, the operators $\Delta_{\nu,\mu}$ and $\Delta_{0,\mu}$ are unitary equivalents and we have
 $$\Delta_{\nu,\mu}= e^{-\frac{i\nu}{2}|z|^2}\Delta_{0,\mu}e^{+\frac{i\nu}{2}|z|^2}.$$
 as in \cite{}.
  \end{proof}

\begin{remark}
A direct proof of Proposition~\ref{Claim6} can be handled using Proposition~\ref{Claim5} and the asymptotic behavior of the confluent hypergeometric function
given by \cite[page 332]{NO1978Fr}
\begin{eqnarray}
{_1F_1}(a; c; x) = \Gamma(c)\set{\frac{(-x)^{-a}}{\Gamma(c-a)} + \frac{e^{x} x^{a-c}}{\Gamma(a)}  } \Big(1+O(\frac 1 x )\Big) ~~
%\mbox{as } x \rightarrow +\infty
\label{AsympF11}
\end{eqnarray}
as $x \rightarrow +\infty$.
This asymptotic behaviour can also be used to show that the radial function $\varphi^{\nu,\mu}_\lambda$ given by \eqref{radialsolution} is
bounded if and only if $\lambda=l$; $l=0,1,2, \cdots $.
\end{remark}

\section{Factorisation of $\Delta_{\nu,\mu}$ and associated Hermite polynomials}

In this section we study the spectral theory of $ \Delta_{\nu,\mu} $ on $L^2(\C^n,d\lambda)$ using the factorisation method. This method %originated by
finds its origin in the works of Dirac \cite{dirac1935principles} and Schr\"odinger \cite{schrodinger-method}, then developed by Infeld and Hull
 \cite{infeld-factorization} in order to solve eigenvalue problems appearing in quantum theory.
Notice for instance that the operator $\Delta_{0,\mu}=L_\mu$ is refereed in physic-mathematical literature as the Landau operator on $\R^{2n}=\C^n$
(or Schr\"odinger  operator on $\R^{2n}$ in the presence of a uniform  magnetic field $\overrightarrow{B}_\mu= \sqrt{-1}d\theta_{\nu,\mu}$)
and for which many of their spectral properties that we are considering go back to Landau's work in 1930 on the Hamiltonian in $\R^{2}=\C$ given by
$$
L_\mu=-4\frac{\partial^2 }{\partial z\bar z} - 2\mu\left(z \frac{\partial }{\partial z} - \bar z \frac{\partial }{\partial \bar z}\right) +\mu^2|z|^2
= 4\left(-\frac{\partial }{\partial  z}+\frac{\mu}{2} \bar z\right) \left( \frac{\partial }{\partial \bar z} +\frac{\mu}{2} z\right) -2\mu I
$$
More generally the Laplacian $\Delta_{\nu,\mu}$ can be rewritten as
$$
\Delta_{\nu,\mu}=- 4\sum_{j=1}^n\left(-\frac{\partial }{\partial  z_j}+\frac{\mu-i\nu}{2} \bar z_j\right) \left( \frac{\partial }{\partial \bar
z_j} +\frac{\mu+i\nu}{2} z_j\right) -2\mu nI .
$$
Hence in view of the above remarks, the spectral properties of $\Delta_{\nu,\mu}$ on $L^2(\C^n,dm)$ or on $\mathcal{C}^\infty(\C^n)$
 can be derived from Landau's work \cite{Landau}. To this end, It will be helpful to define
 $$ \tilde{\Delta}_{\nu,\mu} := -\frac{1}{4} \Delta_{\nu,\mu}. $$
  We also need to define, for $ j=1,2,\cdots,n $, the following first order differential operators
\begin{equation}
a^+_j=-\frac{\partial}{\partial z} + \frac{\mu - i\nu}{2}\overbar{z}_j
= - \e^{ \frac{\mu-i\nu}{2} \modul{z_j}^2} \frac{\partial}{\partial  z_j} \e^{- \frac{\mu-i\nu}{2}\modul{z_j}^2},
\end{equation}
and
\begin{equation}\label{eq:annihilation1}
a^-_j=\frac{\partial}{\partial \overbar z} + \frac{\mu + i\nu}{2}{z}_j = \e^{-\frac{\mu+i\nu}{2}\modul{z_j}^2}\frac{\partial}{\partial \overbar z_j} \e^{
\frac{\mu+i\nu}{2}\modul{z_j}^2}.
\end{equation}
These operators satisfy the commutation relationships $\ent{a^-_j,a^+_k} = n\mu\delta_{j,k}, $
where $ \delta_{j,k} $ is the Kr\"onecker symbol. They are linked to the Laplacian $ \tilde{\Delta}_{\nu,\mu} $ through
\begin{equation}
\sum\limits_{j=1}^{n}a^+_ja^-_j = \tilde{\Delta}_{\nu,\mu} - \frac{n}{2}\mu \quad \mbox{and} \quad
\sum\limits_{j=1}^{n}a^-_ja^+_j = \tilde{\Delta}_{\nu,\mu} + \frac{n}{2}\mu.
\end{equation}
Moreover, we have the following creation and annihilation equalities
\begin{align*}
\tilde{\Delta}_{\nu,\mu} a^+_j = a^+_j (\tilde{\Delta}_{\nu,\mu} + \mu) \quad \mbox{and} \quad
\tilde{\Delta}_{\nu,\mu} a^-_j = a^-_j (\tilde{\Delta}_{\nu,\mu} - \mu)
\end{align*}
and allow the determination of the eigenvalues and eigenvectors of $ \tilde{\Delta}_{\nu,\mu} $. Indeed, if $ \psi $ is an eigenvector of $ \tilde{\Delta}_{\nu,\mu} $ associated to the eigenvalue $ \lambda $ we have the following
\begin{align}
\tilde{\Delta}_{\nu,\mu} (a^+_j \psi) &= a^+_j (\tilde{\Delta}_{\nu,\mu} + \mu) \psi = (\lambda+\mu) a^+_j \psi,  \label{eq:creation2} \\
\tilde{\Delta}_{\nu,\mu} (a^-_j \psi) &= a^-_j (\tilde{\Delta}_{\nu,\mu} - \mu) \psi = (\lambda-\mu) a^-_j \psi. \label{eq:annihilation2}
\end{align}
Thus, we need only to know those associated to the lowest eigenvalue. In fact, since $ \tilde{\Delta}_{\nu,\mu} $ is positive semi-definite,
 all the eigenvalues are real and nonnegative. Moreover, from symmetry and ellipticity of $ \tilde{\Delta}_{\nu,\mu} $ we know that $\tilde{\Delta}_{\nu,\mu} $ has an infinite sequence of nonnegative eigenvalues (see for example \cite{kuwabara-spectra}):
\[ 0 \leq \lambda_0 <\lambda_1<\cdots \uparrow \infty. \]
Therefore, if $ \psi_0 $ is an eigensolution associated to $ \lambda_0 $, we have necessary $a^-_j \psi_0 = 0$ for every $j=1,2,\cdots,n$,
thanks to \eqref{eq:annihilation2}. This implies, by using the second expression in \eqref{eq:annihilation1}, that $\psi_0(z) = \e^{-\frac{\mu+i\nu}{2}\modul{z}^2} f(z)$, where $f$ is any arbitrary holomorphic function.
Consequently,
$$\mathbf{A}_0 =Ker (a^-_j )=
 \overbar{\operatorname{span}}\set{z^m \e^{-\frac{\mu+i\nu}{2}\modul{z}^2}; \ m\in (\Z^+)^n}.$$
Here we have used the multivariate notation $z^m$ for given multi-index $m:= (m_1,m_2,\cdots,m_n) $ to mean $ z^m:=z_1^{m_1} z_2^{m_2}\cdots z_n^{m_n} .$
Making use of the creation operators leads to the following family of multi-indexed functions $ h_{r,s}^{\nu,\mu} $; $ r=(r_1,\cdots,r_n); s=(s_1,\cdots,s_n) $,
\begin{align}
 h_{r,s}^{\nu,\mu}  &= (a^+_1)^{r_1} \cdots (a^+_n)^{r_n} (z^s \e^{-\frac{\mu+i\nu}{2}\modul{z}^2}) \nonumber\\
&= (-1)^{\abs{r}} \e^{\frac{\mu-i\nu}{2} \modul{z}^2} D^r_z \e^{- \frac{\mu-i\nu}{2}\modul{z}^2} (z^s \e^{-\frac{\mu+i\nu}{2}\modul{z}^2})\nonumber\\
&= (-1)^{\abs{r}+ \abs{s}} \mu^{-\abs{s}} \e^{\frac{\mu-i\nu}{2} \modul{z}^2} D^r_z D^s_{\overbar{z}} \e^{-\mu \modul{z}^2},
\label{eq:hermitbasis}
\end{align}
where $\abs{m}$ and $m!$ stand for $ \abs{m}:=m_1+\cdots+m_2 $ and $ m! := m_1!\cdots m_n! $ respectively, and $D^m_{z}$ and $D^m_{\overbar z}$  are defined by
\begin{equation*}
D^m_{z}:= \dfrac{\partial^{\abs{m}}}{\partial z_1^{m_1}\cdots \partial z_n^{m_n}} \quad \mbox{ and } \quad D^m_{\overbar z}:= \dfrac{\partial^{\abs{m}}}{\partial \overbar z_1^{m_1}\cdots \partial \overbar z_n^{m_n}}.
\end{equation*}
According to the above discussion, $  h_{r,s}^{\nu,\mu}  $ are eigensolutions associated to the eigenvalue $ \frac{n}{2}\mu+\abs{r}\mu $ and $\lambda_l :=\mu(\frac{n}{2}+l)$; $l=0,1,2,\cdots ,$ are, indeed, eigenvalues of $ \tilde{\Delta}_{\nu,\mu} $.
 %But, until now, there is no guarantee that there is no eigenvalue other than $\lambda_l$; $ l=0,1,2,\cdots $.
The following proposition shows that $ \set{\lambda_l:=\mu(\frac{n}{2}+l); \, l=0,1,2,\cdots } $ are the only eigenvalues of $ \tilde{\Delta}_{\nu,\mu} $.
\begin{proposition} The $  h_{r,s}^{\nu,\mu}  $ form a complete orthogonal system in the Hilbert space $ L^2(\C^n,d\lambda) $. Moreover, we have the following decomposition $L^2(\C^n,d\lambda) = \bigoplus_{l=1}^{\infty} \mathbf{A}_l$, where
$$
\mathbf{A}_l :=
%Ker (\delta_{\nu,\mu}- \lambda_l) =
 \overbar{\operatorname{span}}\set{ h_{r,s}^{\nu,\mu} ;\ \ r,s \in (\Z^+)^n,\ \abs{r}=l }.
$$
\end{proposition}

\begin{proof} The identity \eqref{eq:hermitbasis} shows that $h_{r,s}^{\nu,\mu} $, up to $e^{i\nu|z|^2/2}$,
 are essentially the high-dimensional analogue of the univariate complex Hermite functions
 \begin{equation*}
h_{m,n}^{\sigma}(\xi,\overbar{\xi}):= (-1)^{m+n} \e^{ \sigma |\xi|^2)} \frac{\partial^{m+n}}{\partial \xi^{m} \partial \overbar{\xi}^{n}} \e^{- \sigma|\xi|^2}; \, \xi\in \C, \sigma>0,
\end{equation*}
considered in \cite{Ito1952,
%\cite{intissar-cauchy},
ghanmi-class}.
The main idea of the proof is then to separate the variable $ z $ in the expression \eqref{eq:hermitbasis} into its components $ z_j $ to get
\begin{align*}
 h_{r,s}^{\nu,\mu} (z,\overbar{z}) &=\mu^{-\abs{s}} \e^{ \frac{-\mu-i\nu}{2} \modul{z}^2}  \prod_{j=1}^{n} (-1)^{r_j+s_j}  \e^{\mu |z_j|^2} \frac{\partial^{r_j+s_j}}{\partial z^{r_j}_j \partial \overbar{z}^{s_j}_j} \e^{-\mu |z_j|^2}\\
 & =  \mu^{-\abs{s}} \e^{ \frac{-\mu-i\nu}{2} \modul{z}^2}  \prod_{j=1}^{n} h_{r_j,s_j} (\sqrt{\mu}z_j,\sqrt{\mu}\overbar{z}_j).
\end{align*}
\end{proof}
This implies that the eigenvalues of $ \Delta_{\nu,\mu} = -4 \tilde{\Delta}_{\nu,\mu} $  are
$$ \set{-2\mu({n}+2l); \, l=0,1,2,\cdots }, $$
which coincide with the results in Section~\ref{section:spectral1}. Notice as well that $ A^2_l(\Delta_{\nu,\mu}) $ and $ \mathbf{A}_l $ refer to the same set.
\begin{remark}
The Hermite functions $h_{r,s}^{\nu,\mu}  $ are given explicitly by
\begin{equation}
  h_{r,s}^{\nu,\mu}  =\mu^{-\abs{s}} \e^{\frac{-\mu-i\nu}{2} \modul{z}^2} \sum_{\abs{k}=0}^{\min(r,s)} \frac{(\sqrt{\mu}^{\abs{r}+\abs{s}-2\abs{k}})
(-1)^{\abs{r-s}} r! s!}{k! (r-k)! (s-k)!} z^{s-k}\overbar{z}^{r-k},
\end{equation}
where $ r=(r_1,\cdots,r_n) $, $ s=(s1,\cdots,s_n) $ and $ \min(r,s):=(\min(r_1,s_1),\cdots,\min(r_n,s_n)) $.
\end{remark}

\section{Concluding remarks}

The consideration of the unitary transformations $T^{\nu,\mu}_g$; $g\in G= U(n) \ltimes  \C^n$, and the  $T^{\nu,\mu}_g$-invariance property satisfied by
the magnetic Laplacian $\Delta_{\nu,\mu}$ give rise to new class of automorphic functions associated to the automorphic factor $j^{\nu,\mu}(g,z)$ when we restrict $g$ to belong in a full-rank discrete subgroup $\Gamma$ of $G$. We call them automorphic functions of bi-weight $(\nu,\mu)$.
The considered $\Delta_{\nu,\mu}$ leaves invariant this space and therefore the corresponding eigenvalue problem is well defined. Thus, a detailed description of the spectral properties of $\Delta_{\nu,\mu}$ when acting on bi-weighted automorphic functions  with respect to any discrete subgroup of $(\C^n,+)$ (not necessary of full-rank) is of great interest. We hope to focuss on this in a near future.
We conclude, by noting that the particular case $\nu=0$ and $\Gamma=\{1\} \ltimes  \C^n$, these functions reduce further the classical one studied in \cite{ghanmi2008landau}.

\quad

%\noindent{\bf Competing Interests:}
%The authors declare that there is no conflict of interests
%regarding the publication of this article and regarding the
%funding that they have received.

\bibliographystyle{abbrv} %abbrv acmlarge
%\bibliography{sublaplacianRef}

\end{document}